\documentclass[11pt,reqno]{amsart}

\usepackage[T1]{fontenc}
\usepackage[utf8]{inputenc}
\usepackage{amsmath,amsfonts,amssymb,amsthm}
\usepackage{mathrsfs}
\usepackage[margin=1.1in]{geometry}
\usepackage[colorlinks=true,linkcolor=blue,citecolor=blue,urlcolor=blue]{hyperref}

\numberwithin{equation}{section}

\theoremstyle{plain}
\newtheorem{theorem}{Theorem}
\newtheorem{lemma}{Lemma}

\newtheorem{thmA}{Theorem}

\theoremstyle{definition}
\newtheorem{remark}{Remark}

\renewcommand{\le}{\leqslant}
\renewcommand{\ge}{\geqslant}

\begin{document}

\title[On the stability of the independence number]
      {On the Stability of the Independence Number\\ in Random Distance Graphs}

\author{V.~A.~Pokhachevskiy}
\address{Lomonosov Moscow State University, Moscow, Russia}
\email{pokhachevskiy@gmail.com}

\author{A.~M.~Raigorodskii}
\address{Moscow Institute of Physics and Technology (National Research
University), Dolgoprudny, Moscow Region, Russia;
Lomonosov Moscow State University, Moscow, Russia;
Caucasus Mathematical Center, Adyghe State University, Maikop, Russia;
Buryat State University, Institute for Mathematics and Informatics,
Ulan-Ude, Russia}
\email{mraigor@yandex.ru}


\thanks{This paper has been submitted to \emph{Matematicheskie Zametki}
(\emph{Mathematical Notes}).}

\date{\today}
\begin{abstract}
We consider a random subgraph $G_p(n,r,<s)$ of the complete distance
graph $G(n,r,<s)$ whose vertices are the $r$-element subsets of the set
$\{1,\dots,n\}$ and whose edges join pairs of subsets that intersect
in fewer than $s$ elements; each edge survives
independently of the others with probability $p$. The independence
number of the graph $G(n,r,<s)$ equals $C_{n-s}^{r-s}$ --- this is the
classical Erd\H{o}s--Ko--Rado theorem. We prove that, for
$r=r(n)\to\infty$, $s=s(n)\to\infty$, $s=o(r)$, $r^2=o(n)$ and
$p\ge 16\,s r^2\ln(n/r)/n$, with probability tending to~1 the
independence number of the random graph $G_p(n,r,<s)$ also equals
$C_{n-s}^{r-s}$, i.e., the Erd\H{o}s--Ko--Rado result is stable under
random sparsification of the graph. Thereby, in the range of parameters
$s\to\infty$, $s=o(r)$, a recent result of Raigorodskii and Karas is
strengthened: the lower bound on the probability $p$ that guarantees
stability is lowered by a factor of about $r/s$.
\end{abstract}

\maketitle


\section{Introduction}
\label{sec:intro}

\subsection{The Erd\H{o}s--Ko--Rado theorem and its stability}

Let $n$, $r$, $s$ be positive integers, $s<r<n$, and let
$[n]=\{1,2,\dots,n\}$. A family $\mathcal{M}$ of $r$-element subsets of
the set $[n]$ is called {\it $s$-intersecting} if any two of its elements
have at least $s$ common elements. Denote by $m(n,r,s)$ the maximum
cardinality of an $s$-intersecting family of $r$-element subsets of the
set $[n]$.

The simplest example of an $s$-intersecting family is constructed as
follows: one fixes $s$ elements of the set $[n]$ and takes all
$r$-element subsets containing them. Such a family is called a {\it
star}; its cardinality equals $C_{n-s}^{r-s}$, whence
$m(n,r,s)\ge C_{n-s}^{r-s}$. The famous Erd\H{o}s--Ko--Rado theorem
(see~\cite{EKR}) states that for all sufficiently large $n$ this bound
is sharp.

\begin{thmA}[Erd\H{o}s, Ko, Rado \cite{EKR}]
\label{thA}
Let $r$ and $s$ be fixed positive integers, $s<r$. Then there exists
$n_0(r,s)$ such that $m(n,r,s)=C_{n-s}^{r-s}$ for all $n\ge n_0(r,s)$.
\end{thmA}

Theorem~\ref{thA} is one of the fundamental results of extremal
combinatorics; it has been generalized and refined many times (see the
survey \cite{DezaFrankl} and the paper \cite{AK}, in which the quantity
$m(n,r,s)$ is found for all values of the parameters). Let us note, in
particular, a stability result that plays a key role in the present
paper: every $s$-intersecting family that is not contained in any star
is substantially smaller than the maximum one. For fixed $r$ and $s$ and
all sufficiently large $n$ this was proved by Frankl \cite{Frankl}; the
best possible bound on $n$, independent of $r$ and $s$, was found by
Ahlswede and Khachatrian \cite{AK1}, which is what makes the result
applicable for growing $r=r(n)$ and $s=s(n)$.

\begin{thmA}[Frankl \cite{Frankl}, Ahlswede, Khachatrian \cite{AK1}]
\label{thB}
Let $n>(r-s+1)(s+1)$, and let $\mathcal{M}$ be an arbitrary
$s$-intersecting family of $r$-element subsets of the set $[n]$. Then
either $\mathcal{M}$ is contained in some star, or
\begin{equation}
\label{eq:frankl}
|\mathcal{M}|\le\max\Bigl\{\,(s+2)\,C_{n-s-2}^{r-s-1}+C_{n-s-2}^{r-s-2},\ \
\sum_{i=1}^{r-s}C_{r-s+1}^{i}\,C_{n-r-1}^{r-s-i}+s\,\Bigr\}.
\end{equation}
(The first quantity under the maximum sign corresponds to the family of
all $r$-sets containing at least $s+1$ elements of $\{1,\dots,s+2\}$;
the second one corresponds to a generalization of the Hilton--Milner
example \cite{HM}. For $r\le2s+1$ the maximum is attained at the first
of them, and for $r>2s+1$ and sufficiently large $n$ at the second one.)
\end{thmA}

\subsection{Graph interpretation and random subgraphs}

In 2012, A.\,M.~Raigorodskii proposed to study the probabilistic
stability of Theorem~\ref{thA} in the language of random graph theory
(see~\cite{BGPR1,BGPR2}). Consider the graph
$G(n,r,<s)=(V(n,r),E(n,r,s))$, where
\begin{displaymath}
V(n,r)=\bigl\{x\subset[n]\colon |x|=r\bigr\},\qquad
E(n,r,s)=\bigl\{\{x,y\}\colon |x\cap y|<s\bigr\}.
\end{displaymath}
The graph $G(n,r,<s)$ is called a {\it complete distance graph}: if one
identifies the $r$-element subsets with their characteristic vectors
from $\{0,1\}^n$, then the adjacency of a pair of vertices means that
their scalar product (and hence the Euclidean distance between them)
belongs to a prescribed set of values. For $s=1$ the graph $G(n,r,<1)$
is the well-known Kneser graph. Graphs of this kind arise naturally in
combinatorial geometry --- in problems on chromatic numbers of spaces
and on the Borsuk problem (see the surveys
\cite{Rai01,Rai5,Rai6,Rai7,Rai8}).

Along with $G(n,r,<s)$, a closely related model is actively studied, in
which the edges are generated by the {\it exact} equality of the scalar
product of the characteristic vectors to a prescribed value: the graph
$G(n,r,s)$ has the same vertex set $V(n,r)$, and its edges join the
pairs of $r$-sets intersecting in exactly $s$ elements. Such graphs are
called {\it Johnson graphs}; for $s=0$ one again obtains the Kneser
graph. The independence numbers, clique numbers and chromatic numbers of
Johnson graphs, as well as the behaviour of these characteristics for
their random subgraphs, are the subject of a large series of papers; a
detailed survey of the current state of this theory is given in the
recent paper \cite{RaiSin} (see especially \S\,3.3, where the
stability of the independence number is treated separately in the
regimes $r>2s+1$ and $r\le2s+1$). Note that in the model of exact
equality even the case of fixed $r$ and $s=1$ was refined only very
recently \cite{Koshelev}. In what follows we work with the model
of strict inequality $G(n,r,<s)$.

Recall that the {\it independence number} $\alpha(G)$ of a graph
$G=(V,E)$ is the maximum cardinality of a set of vertices pairwise not
joined by edges (an independent set). The independent sets of the graph
$G(n,r,<s)$ are exactly the $s$-intersecting families, so that
\begin{displaymath}
\alpha\bigl(G(n,r,<s)\bigr)=m(n,r,s),
\end{displaymath}
and for $n\ge n_0(r,s)$ Theorem~\ref{thA} gives
$\alpha(G(n,r,<s))=C_{n-s}^{r-s}$.

Let now $p=p(n)\in[0,1]$. Consider the {\it random subgraph}
$G_p(n,r,<s)$ --- a random element taking values in the set of all
spanning subgraphs $G=(V(n,r),E)$ of the graph $G(n,r,<s)$, distributed
according to the binomial law
\begin{displaymath}
\mathsf{P}\bigl(G_p(n,r,<s)=G\bigr)=p^{|E|}(1-p)^{|E(n,r,s)|-|E|}.
\end{displaymath}
In other words, each edge of the graph $G(n,r,<s)$ survives
independently of the others with probability $p$. Since the independence
number does not decrease when one passes to a spanning subgraph, one
always has $\alpha(G_p(n,r,<s))\ge\alpha(G(n,r,<s))$. The main question
is to find conditions on $p$ under which, with probability
tending to~1,
\begin{equation}
\label{eq:stability}
\alpha\bigl(G_p(n,r,<s)\bigr)=C_{n-s}^{r-s},
\end{equation}
i.e., the independence number {\it does not increase even after the
deletion of a substantial part of the edges}. It is natural to call the
property \eqref{eq:stability} the stability of the
Erd\H{o}s--Ko--Rado theorem.

Note that a probabilistic setting close in spirit --- on the maximum
$s$-intersecting subfamilies of a random family of $r$-element sets ---
was considered earlier by Balogh, Bohman and Mubayi \cite{BBM}; this
circle of questions was actively developed later as well
(see~\cite{BDDLS}).

The first result on stability in the sense of \eqref{eq:stability} was
obtained by L.\,I.~Bogolubsky, A.\,S.~Gusev, M.\,M.~Pyaderkin and
A.\,M.~Raigorodskii \cite{BGPR1,BGPR2}.

\begin{thmA}[see \cite{BGPR1,BGPR2}]
\label{thC}
Let $r$ be fixed and let $p\in(0,1)$ be fixed. Then there exists a
function $f(n)\to 0$ such that
\begin{displaymath}
\mathsf{P}\Bigl(\bigl|\alpha\bigl(G_{p}(n,r,<1)\bigr)-C_{n-1}^{r-1}\bigr|
<f(n)\,C_{n-1}^{r-1}\Bigr)\to 1,\qquad n\to\infty.
\end{displaymath}
\end{thmA}

Then Bollob\'as, Narayanan and Raigorodskii \cite{BNR} found, for $s=1$,
the sharp threshold probability separating stability from instability,
moreover for growing $r=r(n)$.

\begin{thmA}[Bollob\'as, Narayanan, Raigorodskii \cite{BNR}]
\label{thD}
Put
\begin{displaymath}
p_c(n,r)=\frac{(r+1)\ln n-r\ln r}{C_{n-1}^{r-1}}.
\end{displaymath}
Then for any $r=r(n)$ such that $r\ge2$ and $r=o(n^{1/3})$, and any
fixed $\varepsilon>0$,
\begin{displaymath}
\mathsf{P}\bigl(\alpha(G_p(n,r,<1))=C_{n-1}^{r-1}\bigr)\to
\begin{cases}
1, & p\ge(1+\varepsilon)p_c(n,r),\\
0, & p\le(1-\varepsilon)p_c(n,r),
\end{cases}
\qquad n\to\infty.
\end{displaymath}
\end{thmA}

The restriction $r=o(n^{1/3})$ in Theorem~\ref{thD} was subsequently
removed completely. S.~Das and T.~Tran \cite{DasTran} found the {\it
sharp} threshold probability (i.e., a function $p_0$ such that for
$p\ge(1+\varepsilon)p_0$ the equality \eqref{eq:stability} holds with
probability tending to~$1$, and for $p\le(1-\varepsilon)p_0$ with
probability tending to~$0$; see Theorem~\ref{thBKL}) for $r\le n/C$,
where $C$ is some absolute constant, and for $r\le(1/2-\gamma)n$ with
any fixed $\gamma>0$ they determined it up to a constant factor
depending on~$\gamma$. P.~Devlin and J.~Kahn \cite{DevlinKahn} got rid
of the dependence of this factor on~$\gamma$, having found the threshold
probability up to order of magnitude for all $n\ge2r+2$ (which was a new
result for $r\sim n/2$), and showed that for $n=2r+1$ there exists a
constant $p<1$ such that with probability tending to~$1$ all independent
sets of maximum cardinality in $G_p(n,r,<1)$ are stars. Finally,
J.~Balogh, R.~Krueger and H.~Luo \cite{BKL} found the sharp threshold
probability in the whole range $n\ge2r+1$, thereby completing the Kneser
case.

\begin{thmA}[Balogh, Krueger, Luo \cite{BKL}]
\label{thBKL}
Put
\begin{displaymath}
p_0(n,r)=
\begin{cases}
3/4, & n=2r+1,\\[2mm]
\displaystyle\frac{\ln\bigl(n\,C_{n-1}^{r}\bigr)}{C_{n-r-1}^{r-1}},
& n>2r+1.
\end{cases}
\end{displaymath}
Let $r\ge2$ and let $\varepsilon>0$ be fixed. Then for $n\ge2r+1$ and
$p\ge(1+\varepsilon)\,p_0(n,r)$, with probability tending to~$1$, every
independent set of maximum cardinality in the graph $G_p(n,r,<1)$ is a
star; in particular, $\alpha(G_p(n,r,<1))=C_{n-1}^{r-1}$. For $n=2r+1$
the sharp threshold probability is thus the constant $p_0=3/4$.
\end{thmA}

Note that for $r=o(\sqrt n\,)$ the quantity $p_0(n,r)$ asymptotically
coincides with $p_c(n,r)$ from Theorem~\ref{thD}, so that
Theorem~\ref{thBKL} extends Theorem~\ref{thD} to the whole domain
$n\ge2r+1$, including the case when $r$ grows linearly in~$n$.

Let us return to the general case $s\ge1$. M.\,M.~Pyaderkin
\cite{Pyad15} proved stability for arbitrary fixed $r>s$ and $p=1/2$:
\begin{displaymath}
\mathsf{P}\bigl(\alpha(G_{1/2}(n,r,<s))=C_{n-s}^{r-s}\bigr)\to 1,\qquad
n\to\infty;
\end{displaymath}
for the Johnson graphs $G(n,r,s)$ with fixed $r>2s+1$, the threshold
probability for the stability of the independence number was found up to
order of magnitude in \cite{Pyad19}.

For the graph $G(n,r,<s)$ with arbitrary fixed $r$ and $s$, a complete
picture analogous to Theorem~\ref{thD} was obtained by P.\,A.~Ogarok and
A.\,M.~Raigorodskii \cite{OgRai}.

\begin{thmA}[Ogarok, Raigorodskii \cite{OgRai}]
\label{thOg}
Let $r>s\ge1$ be fixed positive integers and let $\varepsilon>0$ be
fixed. Put
\begin{displaymath}
p_1(n)=\frac{2s\,C_r^s\,\ln n}{C_{n-s}^{r-s}},\qquad
p_2(n)=\frac{(1-\varepsilon)(r+s)\ln n}{C_{n-s}^{r-s}}.
\end{displaymath}
Then for $p=p_1(n)$, with probability tending to~$1$, one has
$\alpha(G_p(n,r,<s))=C_{n-s}^{r-s}$, and moreover all independent sets
of maximum cardinality are stars, while for $p=p_2(n)$, with probability
tending to~$1$, one has $\alpha(G_p(n,r,<s))\ge C_{n-s}^{r-s}+1$.
(The first assertion is proved in \cite{OgRai} under the additional
restriction $r>3$; by monotonicity it is also true for all
$p\ge p_1(n)$, and the second one for all $p\le p_2(n)$.)
\end{thmA}

Thus, for fixed $r$ and $s$ the threshold probability for stability has
order $\ln n\big/C_{n-s}^{r-s}$: the upper and lower bounds of
Theorem~\ref{thOg} differ only by the constant factor
$2s\,C_r^s/(r+s)$. In essence, Theorem~\ref{thOg} is a transfer of
Theorem~\ref{thD} to the model $G(n,r,<s)$: the scheme of the proof is
the same, and for $s=1$ both theorems give one and the same threshold up
to the factor $2r/(r+1)$. We emphasize that for fixed $r$ and $s$ the
threshold $\ln n\big/C_{n-s}^{r-s}$ decreases as $n^{-(r-s)}$ up to a
logarithmic factor, where the exponent $r-s$ may be arbitrarily large, whereas
all bounds known for growing parameters decrease more slowly than
$n^{-1}$ (see Theorem~\ref{thE} and Theorem~\ref{th:main} below, as well
as Remark~\ref{zam:ogarok}).

The passage to {\it growing} parameters $r=r(n)$, $s=s(n)$ and
decreasing probability $p=p(n)$ was started by V.\,S.~Karas,
P.\,A.~Ogarok and A.\,M.~Raigorodskii \cite{KOR} and was brought to the
following statement in the paper of A.\,M.~Raigorodskii and
V.\,S.~Karas \cite{RaiKaras}.

\begin{thmA}[Raigorodskii, Karas \cite{RaiKaras}]
\label{thE}
Let $r=r(n)$ and $s=s(n)$ be such that $r-s\ge2$, let the probability
$p=p(n)$ be constant or tend to zero, and let $Q:=1/(-\ln(1-p))$. If
$Qr^3=o(n/\ln n)$, then
\begin{displaymath}
\mathsf{P}\bigl(\alpha(G_p(n,r,<s))=C_{n-s}^{r-s}\bigr)\to 1,\qquad n\to\infty.
\end{displaymath}
(For $r-s\in\{1,2\}$, other conditions on the parameters, in a number of
cases less restrictive, are also given in \cite{RaiKaras}; for $r-s=1$
only they are applicable.)
\end{thmA}

For $p\to0$ we have $Q=(1+o(1))/p$, and therefore the condition
$Qr^3=o(n/\ln n)$ means, in essence, that $p\gg r^3\ln n/n$. The main
result of the present paper is that in the range of
parameters $s\to\infty$, $s=o(r)$ this threshold for $p$ can be lowered
by a factor of about $r/s$.

\subsection{The main result}

\begin{theorem}
\label{th:main}
Let $r=r(n)\to\infty$ and $s=s(n)\to\infty$ as $n\to\infty$, where
\begin{displaymath}
s=o(r),\qquad r^2=o(n),
\end{displaymath}
and let $p=p(n)\in[0,1]$ satisfy the inequality
\begin{equation}
\label{eq:pbound}
p\ge p_0(n):=\frac{16\,s r^2\ln(n/r)}{n}.
\end{equation}
Then
\begin{displaymath}
\mathsf{P}\bigl(\alpha(G_p(n,r,<s))=C_{n-s}^{r-s}\bigr)\to 1,
\qquad n\to\infty.
\end{displaymath}
\end{theorem}

Let us make several comments.

\begin{remark}
\label{zam:p0}
The condition \eqref{eq:pbound} can be satisfied only if $p_0(n)\le1$,
i.e., for $16\,sr^2\ln(n/r)\le n$. If in addition $sr^2=o(n/\ln n)$,
then
\begin{displaymath}
p_0(n)=\frac{16\,sr^2\ln(n/r)}{n}\le\frac{16\,sr^2\ln n}{n}\to 0,
\qquad n\to\infty,
\end{displaymath}
and Theorem~\ref{th:main} guarantees stability for probabilities $p$
tending to zero, i.e., under the deletion of \textit{almost all} edges
of the graph $G(n,r,<s)$. For $sr^2\asymp n/\ln n$ the theorem is
applicable only for $p$ bounded away from zero: for example, for
$sr^2=n/(32\ln n)$ we have $p_0(n)\le1/2$, so that the assertion is true
for $p=1/2$ --- but then the fraction of deleted edges does not tend to
one. Note at the same time that for $s=1$ the threshold probability from
Theorem~\ref{thD} is considerably smaller than the quantity $p_0$: the
bound \eqref{eq:pbound} decreases more slowly than $n^{-1}$, whereas
$p_c(n,r)$ decreases practically as $\bigl(C_{n-1}^{r-1}\bigr)^{-1}$.
The question of finding the threshold probability for growing $s$
remains open.
\end{remark}

\begin{remark}
Let us compare Theorem~\ref{th:main} with Theorem~\ref{thE}. In the
domain of parameters common to them ($s\to\infty$, $s=o(r)$, $p\to0$)
the condition $Qr^3=o(n/\ln n)$ requires $p\gg r^3\ln n/n$, whereas the
bound \eqref{eq:pbound} allows $p\asymp sr^2\ln n/n$, i.e.,
substantially smaller values of $p$ --- the gain is of order
$r/s\to\infty$. Moreover, the condition $r^2=o(n)$ is substantially
milder as regards the admissible growth of $r$ than $Qr^3=o(n/\ln n)$.
On the other hand, Theorem~\ref{thE} requires neither $s\to\infty$ nor
$s=o(r)$, so that the results do not cover each other.
\end{remark}

\begin{remark}
\label{zam:ogarok}
A comparison with Theorem~\ref{thOg} shows that both the bound
\eqref{eq:pbound} and the bound of Theorem~\ref{thE} are rather far from
the true threshold: for fixed $r$ and $s$ the threshold has order
$\ln n\big/C_{n-s}^{r-s}$, whereas \eqref{eq:pbound} decreases more
slowly than $n^{-1}$. The reason for this gap is methodological.
Theorems~\ref{thE} and~\ref{th:main} rely on a crude first-moment
bound, but on the other hand they work for any growth of $r$ and $s$; in
\cite{BNR,OgRai}, on the contrary, a fine analysis of the structure of
near-maximum independent sets is used, which relies essentially on the
parameters being fixed. We believe that the method of \cite{OgRai}
admits an extension at least to slowly growing $r$ and $s$, and that
such an extension would substantially improve both Theorem~\ref{thE} and
Theorem~\ref{th:main}; this seems to us the most promising direction for
further work. Until then, however, Theorems~\ref{thE} and~\ref{th:main}
remain the only results covering the range of rapidly growing $r$
and~$s$, and Theorem~\ref{th:main} gives in this range the best known
bound for $s\to\infty$, $s=o(r)$.
\end{remark}

\begin{remark}
The conditions $s=o(r)$ and $r^2=o(n)$ are quite natural for the
technique employed: they imply, in particular, the inequality
$n>(r-s+1)(s+1)$, which ensures the applicability of
Theorem~\ref{thB}; moreover, for $r^2=o(n)$ the asymptotics of the
binomial coefficients take a simple form. We have not optimized the
constant $16$ in \eqref{eq:pbound}: a slight modification of the proof
allows one to decrease it, but obtaining the optimal dependence
apparently requires new ideas.
\end{remark}

\begin{remark}
In the range of parameters under consideration the equality
$\alpha(G(n,r,<s))=C_{n-s}^{r-s}$ holds for all sufficiently large $n$:
a family of cardinality greater than $C_{n-s}^{r-s}$ is not contained in
a star, while the right-hand side of \eqref{eq:frankl} is much smaller
than $C_{n-s}^{r-s}$ (see Lemma~\ref{lem:asymp}\,(d) below). Therefore,
Theorem~\ref{th:main} indeed asserts stability: with high probability
the independence number of the random subgraph {\it coincides} with the
independence number of the original graph.
\end{remark}

The rest of the paper is organized as follows. In \S\,\ref{sec:prelim}
we introduce the notation and prove auxiliary results, and in
\S\,\ref{sec:proof} we prove Theorem~\ref{th:main}.
\section{Notation and auxiliary results}
\label{sec:prelim}

Throughout what follows we assume that the hypotheses of
Theorem~\ref{th:main} are satisfied: $r\to\infty$, $s\to\infty$,
$s=o(r)$, $r^2=o(n)$. All asymptotics and all signs $o(\,\cdot\,)$ refer
to $n\to\infty$; all statements ``for all sufficiently large $n$'' are
used without further stipulation. For brevity we denote by $G$ the graph
$G(n,r,<s)$, and by $G_p$ its random subgraph $G_p(n,r,<s)$. Put
\begin{displaymath}
M:=C_{n-s}^{r-s},\qquad M':=\frac{4Mr^2}{n},\qquad
L:=s\ln\frac{n}{r}.
\end{displaymath}

Throughout, $C_n^r$ denotes the binomial coefficient $\binom{n}{r}$.

For an $s$-element set $W\subset[n]$ we denote by
$S_W:=\{x\in V(n,r)\colon x\supseteq W\}$ the star centered at $W$;
$|S_W|=M$, and every star is an independent set of the graph $G$.
For $A\subseteq V(n,r)$ we denote by $e(A)$ the number of edges of the
graph $G$ both of whose endpoints lie in $A$, and by $\alpha(A)$ the
maximum cardinality of a subset of $A$ independent in $G$.

We note at once several simple consequences of the hypotheses of the
theorem. First, $r^2=o(n)$ by hypothesis, and hence for large $n$
\begin{equation}
\label{eq:lnnr}
r\le\sqrt{n},\qquad\text{and therefore}\qquad
\ln\frac{n}{r}\ge\frac{1}{2}\ln n .
\end{equation}
Second, $rs\le r^2 = o(n)$, and hence $n>(r-s+1)(s+1)$, so that
Theorem~\ref{thB} is applicable. Finally,
$L=s\ln(n/r)\to\infty$.

\begin{lemma}
\label{lem:asymp}
As $n\to\infty$ the following relations hold:
\begin{itemize}
\item[(a)] $\ln C_n^r-\ln C_{n-s}^{r-s}=(1+o(1))\,s\ln(n/r)$;
\item[(b)] $\ln M=(1+o(1))\,(r-s)\ln(n/r)$; in particular,
$M\,(r^2/n)^3\to\infty$;
\item[(c)] $r\,C_n^{r-s-1}=(1+o(1))\,M\,\dfrac{r(r-s)}{n}$;
in particular, $r\,C_n^{r-s-1}<M'/2$;
\item[(d)] the right-hand side of \eqref{eq:frankl}, which we denote
by $\varphi(n,r,s)$, satisfies the inequality
$\varphi(n,r,s)<M'/2$.
\end{itemize}
\end{lemma}

\begin{proof}
(a) Let $a$, $b$ be positive integers such that $b\to\infty$ and
$b^2=o(a)$. Then
\begin{displaymath}
C_a^b=\frac{a^b}{b!}\prod_{i=1}^{b-1}\Bigl(1-\frac{i}{a}\Bigr)
=\frac{a^b}{b!}\,e^{-\frac{b^2}{2a}(1+o(1))}
=(1+o(1))\,\frac{a^b}{b!},
\end{displaymath}
and by Stirling's formula
\begin{equation}
\label{eq:stirling}
\ln C_a^b=b\ln\frac{a}{b}+b+O(\ln b).
\end{equation}
The condition $b^2=o(a)$ is satisfied for both pairs $(a,b)=(n,r)$ and
$(a,b)=(n-s,r-s)$, since $r^2=o(n)$. Further, using the expansion of the
logarithm and the relations $s=o(r)$, $rs=o(n)$, we obtain
\begin{multline}
\label{eq:middle}
(r-s)\ln\frac{n-s}{r-s}
=(r-s)\Bigl(\ln\frac{n}{r}+\ln\Bigl(1-\frac{s}{n}\Bigr)
-\ln\Bigl(1-\frac{s}{r}\Bigr)\Bigr)=\\
=(r-s)\ln\frac{n}{r}
+(r-s)\Bigl(-\frac{s}{n}+O\Bigl(\frac{s^2}{n^2}\Bigr)
+\frac{s}{r}+O\Bigl(\frac{s^2}{r^2}\Bigr)\Bigr)
=(r-s)\ln\frac{n}{r}+s+o(s).
\end{multline}
Hence and from \eqref{eq:stirling}
\begin{multline*}
\ln C_n^r-\ln C_{n-s}^{r-s}
=r\ln\frac{n}{r}-(r-s)\ln\frac{n-s}{r-s}+s+O(\ln r)=\\
=s\ln\frac{n}{r}+O(s^2/r)+O(\ln r)+o(s)
=(1+o(1))\,s\ln\frac{n}{r},
\end{multline*}
where in the last step we have used the relations $s^2/r=o(s)$ and
$\ln r\le\ln n\le 2\ln(n/r)=o(s\ln(n/r))$ (see~\eqref{eq:lnnr}).

(b) From \eqref{eq:stirling} and \eqref{eq:middle}
\begin{displaymath}
\ln M=(r-s)\ln\frac{n}{r}+s+o(s)+(r-s)+O(\ln r)
=(1+o(1))\,(r-s)\ln\frac{n}{r},
\end{displaymath}
since $(r-s)+s+O(\ln r)=O(r)=o\bigl((r-s)\ln(n/r)\bigr)$. Further, by
\eqref{eq:lnnr} and $s=o(r)$
\begin{displaymath}
\ln\Bigl(M\Bigl(\frac{r^2}{n}\Bigr)^3\Bigr)
=(1+o(1))(r-s)\ln\frac{n}{r}-3\ln\frac{n}{r^2}
\ge(1+o(1))\,\frac{r}{2}\cdot\frac{\ln n}{2}-3\ln n\to\infty.
\end{displaymath}

(c) As in part~(a), $C_n^{r-s-1}=(1+o(1))\,n^{r-s-1}/(r-s-1)!$ and
$M=(1+o(1))\,(n-s)^{r-s}/(r-s)!$, whence
\begin{multline*}
\frac{r\,C_n^{r-s-1}}{M}
=(1+o(1))\,\frac{r(r-s)}{n}
\Bigl(\frac{n}{n-s}\Bigr)^{r-s}
=\\
=(1+o(1))\,\frac{r(r-s)}{n}\,e^{O(rs/n)}
=(1+o(1))\,\frac{r(r-s)}{n}\,,
\end{multline*}
since $rs=o(n)$. In particular,
$r\,C_n^{r-s-1}\le(1+o(1))\,Mr^2/n<2Mr^2/n=M'/2$.

(d) We estimate both quantities under the maximum sign in
\eqref{eq:frankl}. For the first of them, since $s+3\le r$ for large $n$
(recall that $s=o(r)$), by part~(c) we have
\begin{displaymath}
(s+2)\,C_{n-s-2}^{r-s-1}+C_{n-s-2}^{r-s-2}
\le(s+3)\,C_{n}^{r-s-1}\le r\,C_n^{r-s-1}<\frac{M'}{2}.
\end{displaymath}
Let us pass to the second one. Put $t_i:=C_{r-s+1}^{i}\,C_{n-r-1}^{r-s-i}$.
Then for $1\le i\le r-s-1$
\begin{displaymath}
\frac{t_{i+1}}{t_i}
=\frac{r-s+1-i}{i+1}\cdot\frac{r-s-i}{n-2r+s+i}
\le\frac{(r-s)^2}{n-2r}\le\frac{2r^2}{n}=:q,
\end{displaymath}
where $q=o(1)$. Consequently, the sum in the definition of $\varphi$ is
majorized by a geometric progression:
\begin{displaymath}
\sum_{i=1}^{r-s}t_i\le\frac{t_1}{1-q}
=(1+o(1))\,(r-s+1)\,C_{n-r-1}^{r-s-1}
\le(1+o(1))\,(r-s+1)\,C_{n}^{r-s-1}.
\end{displaymath}
By part~(c), $C_n^{r-s-1}=(1+o(1))\,M(r-s)/n$, and hence
\begin{displaymath}
\sum_{i=1}^{r-s}t_i+s\le(1+o(1))\,\frac{M(r-s+1)(r-s)}{n}+s
\le(1+o(1))\,\frac{Mr^2}{n}<\frac{2Mr^2}{n}=\frac{M'}{2},
\end{displaymath}
where we have taken into account that $s=o(Mr^2/n)$ (the quantity
$Mr^2/n$ tends to infinity by part~(b)).
\end{proof}

The next lemma is the key structural observation: a large independent
set must lie in a star, and every vertex outside this star has many
neighbors inside it.

\begin{lemma}
\label{lem:star}
Let $n$ be sufficiently large, let $A\subseteq V(n,r)$, and let
$B\subseteq A$ be a subset of $A$ of maximum cardinality independent in
$G$, and suppose that $|B|>M'$. Then:
\begin{itemize}
\item[(i)] there exists a star $S_W\supseteq B$;
\item[(ii)] no vertex $u\in A\setminus B$ belongs to $S_W$, and each
such vertex is joined by an edge of the graph $G$ to at least
$|B|-r\,C_n^{r-s-1}$ vertices of the set $B$.
\end{itemize}
\end{lemma}

\begin{proof}
(i) The set $B$ is independent in $G$, i.e., it is an $s$-intersecting
family. For large $n$ we have $n>(r-s+1)(s+1)$, and hence by
Theorem~\ref{thB} either $B$ is contained in a star, or
$|B|\le\varphi(n,r,s)$. The latter is impossible: by
Lemma~\ref{lem:asymp}\,(d), $\varphi(n,r,s)<M'/2<|B|$.

(ii) Let $u\in A\setminus B$. If we had $u\in S_W$, then the set
$B\cup\{u\}\subseteq A$ would be independent in $G$ (a star is
independent) and strictly larger than $B$, which contradicts the
maximality of $B$. Hence $u\notin S_W$, i.e., $|u\cap W|\le s-1$.

Let us estimate the number of vertices $w\in S_W$ that are {\it not
adjacent} to $u$ in the graph $G$, i.e., such that $|w\cap u|\ge s$.
Since $|w\cap u\cap W|\le s-1$, every such vertex $w$ must contain,
besides all elements of $W$, at least one element of the set
$u\setminus W$. The number of such $w$ does not exceed
$|u\setminus W|\cdot C_{n}^{r-s-1}\le r\,C_n^{r-s-1}$
(we choose an element $x\in u\setminus W$, and then arbitrarily the
remaining $r-s-1$ elements of the vertex $w\supseteq W\cup\{x\}$). Hence
$u$ is adjacent in $G$ to at least $|B|-r\,C_n^{r-s-1}$ vertices of $B$.
\end{proof}

We shall also need a classical corollary of Tur\'an's theorem (see, for
example, \cite{AlonSpencer}): for any set of vertices
$A\subseteq V(n,r)$
\begin{equation}
\label{eq:turan}
e(A)\ \ge\ \frac{|A|}{2}\Bigl(\frac{|A|}{\alpha(A)}-1\Bigr).
\end{equation}

Finally, we give the standard bounds that we shall use when counting the
number of sets: for positive integers $N\ge k\ge 1$
\begin{equation}
\label{eq:binom}
C_N^k\le\Bigl(\frac{eN}{k}\Bigr)^{k},
\qquad
\sum_{i=0}^{m}C_N^i\le(m+1)\Bigl(\frac{eN}{m}\Bigr)^{m}
\quad(1\le m\le N/2).
\end{equation}
\section{Proof of Theorem~\ref{th:main}}
\label{sec:proof}

\subsection{Outline of the proof}

As was noted in the introduction, $\alpha(G_p)\ge\alpha(G)=M$ for all
sufficiently large $n$, so it suffices to prove that
\begin{displaymath}
\mathsf{P}\bigl(\alpha(G_p)\ge M+1\bigr)\to 0,\qquad n\to\infty.
\end{displaymath}
The event $\{\alpha(G_p)\ge M+1\}$ means that there exists a set
$A\subset V(n,r)$ of cardinality $M+1$ which is independent in $G_p$.
Classifying such sets according to the value of $\alpha(A)$, we bound
the probability of this event by a sum of three terms:
\begin{equation}
\label{eq:split}
\mathsf{P}\bigl(\alpha(G_p)\ge M+1\bigr)\le P_1+P_2+P_3,
\end{equation}
where
\begin{align*}
P_1&:=\mathsf{P}\bigl(\exists\,A\colon|A|=M+1,\ A\text{ is independent in }G_p,\
\alpha(A)\le M'\bigr),\\
P_2&:=\mathsf{P}\bigl(\exists\,A\colon|A|=M+1,\ A\text{ is independent in }G_p,\
M'<\alpha(A)\le M-M'\bigr),\\
P_3&:=\mathsf{P}\bigl(\exists\,A\colon|A|=M+1,\ A\text{ is independent in }G_p,\
\alpha(A)>M-M'\bigr).
\end{align*}
(Note that $M'=4Mr^2/n=o(M)$, so that for large $n$ we indeed have
$M'<M-M'$.) In the next three subsections we show that each of the terms
$P_1$, $P_2$, $P_3$ tends to zero.

In what follows we repeatedly use the general bound below. If a set $A$
with $|A|=M+1$ is independent in $G_p$, then none of the $e(A)$ edges of
the graph $G$ inside $A$ survives in $G_p$; the probability of this
equals $(1-p)^{e(A)}\le e^{-p\,e(A)}$. Moreover, by \eqref{eq:binom} and
Lemma~\ref{lem:asymp}\,(a), the number of ways to choose $A$ does not
exceed
\begin{equation}
\label{eq:count}
C_{C_n^r}^{\,M+1}\le\Bigl(\frac{e\,C_n^r}{M+1}\Bigr)^{M+1}
=\exp\Bigl\{(M+1)\bigl(\ln C_n^r-\ln(M+1)+1\bigr)\Bigr\}
=e^{(1+o(1))(M+1)L},
\end{equation}
where, we recall, $L=s\ln(n/r)\to\infty$.

\subsection{Estimating $P_1$: small $\alpha(A)$}

Let $|A|=M+1$ and $\alpha(A)\le M'$. By Tur\'an's inequality
\eqref{eq:turan}
\begin{displaymath}
e(A)\ge\frac{M+1}{2}\Bigl(\frac{M+1}{M'}-1\Bigr),
\end{displaymath}
whence, taking into account $M'=4Mr^2/n$ and condition \eqref{eq:pbound},
\begin{displaymath}
p\,e(A)\ge\frac{M+1}{2}\Bigl(\frac{p\,n}{4r^2}-p\Bigr)
\ge\frac{M+1}{2}\bigl(4L-1\bigr)\ge(M+1)\,(2L-1),
\end{displaymath}
since $p\,n/(4r^2)\ge 16sr^2\ln(n/r)/(4r^2)=4s\ln(n/r)=4L$ and
$p\le1$. Estimating the probability of a union of events by the sum of
their probabilities and using \eqref{eq:count}, we obtain
\begin{displaymath}
P_1\le
\exp\Bigl\{(1+o(1))(M+1)L-(M+1)(2L-1)\Bigr\}
=\exp\Bigl\{-(1+o(1))(M+1)L\Bigr\}\to 0,
\end{displaymath}
since $(M+1)L\to\infty$.

\subsection{Estimating $P_2$: intermediate $\alpha(A)$}

Let now $|A|=M+1$ and $M'<\alpha(A)\le M-M'$. Fix a subset
$B\subseteq A$ of maximum cardinality which is independent in $G$;
then $|B|=\alpha(A)>M'$, and Lemma~\ref{lem:star} is applicable: there
exists a star $S_W\supseteq B$, while every vertex $u\in A\setminus B$
is adjacent in $G$ to at least
\begin{displaymath}
|B|-r\,C_n^{r-s-1}>M'-\frac{M'}{2}=\frac{M'}{2}
\end{displaymath}
vertices of $B$ (here Lemma~\ref{lem:asymp}\,(c) has been used). The
number of vertices in $A\setminus B$ is at least
\begin{displaymath}
M+1-\alpha(A)\ge M+1-(M-M')>M',
\end{displaymath}
and the edges mentioned are distinct for distinct $u$. Consequently,
\begin{displaymath}
e(A)\ge M'\cdot\frac{M'}{2}=\frac{M'^2}{2},
\end{displaymath}
and by condition \eqref{eq:pbound}
\begin{displaymath}
p\,e(A)\ge\frac{16sr^2\ln(n/r)}{n}\cdot\frac{1}{2}
\Bigl(\frac{4Mr^2}{n}\Bigr)^{2}
=128\,L\,M\cdot M\Bigl(\frac{r^2}{n}\Bigr)^{3}.
\end{displaymath}
From this and \eqref{eq:count} we get
\begin{multline*}
P_2\le\exp\Bigl\{(1+o(1))(M+1)L-128\,L\,M\cdot
M\Bigl(\frac{r^2}{n}\Bigr)^{3}\Bigr\}\le\\
\le\exp\Bigl\{-L\,M\Bigl(128\,M\Bigl(\frac{r^2}{n}\Bigr)^{3}
-3\Bigr)\Bigr\}\to 0,
\end{multline*}
since $M(r^2/n)^3\to\infty$ by Lemma~\ref{lem:asymp}\,(b).

\subsection{Estimating $P_3$: large $\alpha(A)$}

Finally, let $|A|=M+1$ and $\alpha(A)>M-M'$. Again fix a subset
$B\subseteq A$ of maximum cardinality which is independent in $G$,
$|B|=\alpha(A)>M-M'>M'$. By Lemma~\ref{lem:star} there exists a star
$S_W\supseteq B$, and since $|B|\le|S_W|=M<|A|$, there exists a vertex
$u\in A\setminus B$ which does not belong to $S_W$ and is adjacent in
$G$ to at least
\begin{displaymath}
|B|-r\,C_n^{r-s-1}\ge M-M'-\frac{M'}{2}\ge M-2M'
=M\Bigl(1-\frac{8r^2}{n}\Bigr)=(1+o(1))\,M
\end{displaymath}
vertices of the set $B$. Since $A\supseteq B\cup\{u\}$ is independent in
$G_p$, none of these edges survives in $G_p$.

Thus, the event $\{\exists\,A\colon|A|=M+1,\ A$ is independent in
$G_p$, $\alpha(A)>M-M'\}$ implies the existence of a star $S_W$, of a set
$B\subseteq S_W$ with $|S_W\setminus B|<M'$ and of a vertex $u\notin S_W$
which is joined in $G_p$ to none of the at least $(1+o(1))M$ vertices of
the set $B$ that are adjacent to $u$ in $G$. For a fixed triple
$(W,B,u)$ the probability of the latter event does not exceed
$(1-p)^{(1+o(1))M}\le e^{-(1+o(1))pM}$. As for the number of such
triples, by \eqref{eq:binom} it does not exceed
\begin{displaymath}
C_n^{s}\cdot\sum_{i=0}^{\lceil M'\rceil-1}C_{M}^{i}\cdot C_n^{r}
\le C_n^{s}\cdot (M'+1)\Bigl(\frac{eM}{M'}\Bigr)^{M'}\cdot C_n^{r}.
\end{displaymath}
Consequently,
\begin{equation}
\label{eq:P3}
\ln P_3\le s\ln n+\ln(M'+1)
+M'\Bigl(1+\ln\frac{M}{M'}\Bigr)+r\Bigl(1+\ln\frac{n}{r}\Bigr)
-(1+o(1))\,p\,M .
\end{equation}
We show that the leading term on the right-hand side of \eqref{eq:P3} is
the last one. By condition \eqref{eq:pbound}
\begin{displaymath}
p\,M\ge\frac{16sr^2\ln(n/r)}{n}\,M=4L\,M'.
\end{displaymath}
Let us compare the remaining terms with it.
\begin{itemize}
\item[1)] Since $M/M'=n/(4r^2)\le n$ and $\ln n\le2\ln(n/r)$
(see~\eqref{eq:lnnr}),
\begin{displaymath}
M'\Bigl(1+\ln\frac{M}{M'}\Bigr)\le M'\,(1+\ln n)\le 3M'\ln\frac{n}{r}
=\frac{3}{s}\,L\,M'=o(L\,M').
\end{displaymath}
\item[2)] By Lemma~\ref{lem:asymp}\,(b), $\ln M\ge(1+o(1))\,
\frac{r}{2}\ln\frac{n}{r}$, hence $M\ge e^{r\ln(n/r)/3}\ge n^{r/6}$,
and therefore $M'=4Mr^2/n\ge M/n\ge n^{r/6-1}\gg r\ln n$. Hence
\begin{displaymath}
r\Bigl(1+\ln\frac{n}{r}\Bigr)\le 2r\ln n=o(M')=o(L\,M').
\end{displaymath}
\item[3)] Similarly, $s\ln n\le 2L=o(LM')$ and
$\ln(M'+1)\le\ln M\le 2r\ln n=o(LM')$.
\end{itemize}
Thus, all positive terms on the right-hand side of \eqref{eq:P3} are
$o(pM)$, and therefore
\begin{displaymath}
\ln P_3\le-(1+o(1))\,p\,M\le-(1+o(1))\,4LM'\to-\infty,
\end{displaymath}
i.e., $P_3\to0$.

Combining the bounds $P_1,P_2,P_3\to0$ with \eqref{eq:split}, we complete
the proof of Theorem~\ref{th:main}. The theorem is proved.

\section{Concluding remarks}
\label{sec:concl}

Theorem~\ref{th:main} gives only a sufficient condition for stability,
and the question of the {\it threshold probability} $p_c=p_c(n,r,s)$,
upon crossing which the stability \eqref{eq:stability} is replaced by
instability, remains open: for $s=1$ the answer is given by
Theorem~\ref{thD}, and it shows that a considerable gap remains between
the true threshold and the bound \eqref{eq:pbound}. It would also be
interesting to weaken the condition $r^2=o(n)$: it is dictated by the
technique used (Theorem~\ref{thB} and the asymptotics of binomial
coefficients) rather than by the essence of the problem.

As was noted in Remark~\ref{zam:ogarok}, the most natural way to find the
threshold for growing parameters is to transfer the technique of
\cite{BNR,OgRai} to the case of $r=r(n)$ and $s=s(n)$ growing
sufficiently slowly. Even such an advance would apparently improve
substantially both Theorem~\ref{thE} and Theorem~\ref{th:main}; we intend
to return to this question in the future. Finally, for the model of exact
equality --- the Johnson graphs $G(n,r,s)$, a survey of results on which
is given in \cite{RaiSin} --- the analogous questions on the stability of
the independence number for growing $r$ and $s$ also remain largely
open.

\medskip

\noindent\textbf{Acknowledgments.} The authors are grateful to
M.~M.~Koshelev for a careful reading of the manuscript and for valuable
comments made during the preparation of this work.

\end{document}